\documentclass{amsart}
\usepackage{mathpazo}
\usepackage{amssymb}
\usepackage{csquotes}
\usepackage{hyperref}
\hypersetup{
colorlinks=true,
linkcolor=blue,
filecolor=blue,
urlcolor=blue,
citecolor=blue,
pdfpagemode=Fullscreen
}
\newtheorem{theorem}{Theorem}[section]
\newtheorem{lemma}[theorem]{Lemma}
\newtheorem*{Acknowledgement}{\textnormal{\textbf{Acknowledgement}}}
\theoremstyle{definition}
\newtheorem{definition}[theorem]{Definition}
\newtheorem{example}[theorem]{Example}

\newtheorem{corollary}[theorem]{Corollary}
\newtheorem{proposition}[theorem]{Proposition}

\newtheorem{remark}[theorem]{Remark}

\numberwithin{equation}{section}
\newcommand{\beqa}{\begin{eqnarray*}}
\newcommand{\eeqa}{\end{eqnarray*}}
\newcommand{\beqn}{\begin{eqnarray}}
\newcommand{\eeqn}{\end{eqnarray}}

\newcounter{cnt1}
\newcounter{cnt2}
\newcounter{cnt3}
\newcommand{\blr}{\begin{list}{$($\roman{cnt1}$)$}
        {\usecounter{cnt1} \setlength{\topsep}{0pt}
                \setlength{\itemsep}{0pt}}}
\newcommand{\bla}{\begin{list}{$($\alph{cnt2}$)$}
        {\usecounter{cnt2} \setlength{\topsep}{0pt}
                \setlength{\itemsep}{0pt}}}
\newcommand{\bln}{\begin{list}{$($\arabic{cnt3}$)$}
        {\usecounter{cnt3} \setlength{\topsep}{0pt}
                \setlength{\itemsep}{0pt}}}
\newcommand{\el}{\end{list}}
\newtheorem{thm}{Theorem}

\newtheorem{Def}[thm]{Definition}

\newtheorem{rem}[thm]{Remark}
\newcommand{\Rem}{\begin{rem} \rm}
\newcommand{\bdfn}{\begin{Def} \rm}
\newcommand{\edfn}{\end{Def}}

\title{Geometric characterization of Generalized Property (II)}
\author[ S. Basu, \ S. Seal ]
			{Sudeshna Basu$^{1}$, Susmita Seal$ ^{2}$ }
		\address{{$^{1}$}   Sudeshna Basu,
		Department of Mathematics and Statistics, 
				Loyola University, 
				Baltimore, MD 21210, USA 
				}
		\email{sudeshnamelody@gmail.com} 
			
		\address {{$^{2}$} Susmita Seal, 
				School of Mathematical Sciences, National Institute of Science Educational and Research Bhubaneswar, An OCC of Homi Bhabha National Institute, P.O. - Jatni, District - Khurda, Odisha - 752050, India}
		\email{susmitaseal1996@gmail.com}

			\date{}
\begin{document}
\maketitle
\begin{abstract}
Let $\mathcal{A}$ be a compatible collection of bounded subsets of a Banach space $X$.
In this paper, we introduce the notion of $\mathcal{A}$-Property (II) and prove that 
 $X$ has $\mathcal{A}$-Property (II) if and only if every $f$ in $S_{X^*}$ is $\mathcal{A}$-semi PC of $B_{X^*}$.
\end{abstract}

\section{Introduction}

One of the most significant consequences of the Hahn–Banach Theorem is that a point lying outside a closed bounded convex set can be separated from the set by a hyperplane. The ball separation property, on the other hand, explores whether such a point can be separated from the set by closed balls as opposed to hyperplanes. While finite dimensional spaces admit an equivalence between these two notions, such an equivalence does not hold in general infinite dimensional Banach spaces.
The origins of ball separation property go back to Mazur \cite{M}, who introduced what is now known as the Mazur Intersection Property (MIP): every bounded closed convex set can be represented as an intersection of closed balls. Since then, several variants of ball separation properties have been investigated. A comprehensive survey of this area can be found in \cite{GJSM}. See also \cite{CL}. 

Among these, Property (II), introduced by Chen and Lin \cite{CL}, has emerged as an important weakening of MIP.
 A Banach space $X$ is said to have Property (II) if every closed bounded convex set $C\subset X$ can be expressed as an intersection of closed convex hulls of finitely many balls. 
They also obtained a remarkable geometric characterization: $X$ has Property (II) if and only if the set of $w^*$-Points of Continuity of $B_{X^*}$ is dense in $S_{X^*}$. A similar type of characterization for MIP, in terms of $w^*$-denting points, was earlier established by Giles, Gregory, and Sims \cite{GGS}. The importance of such characterizations lies in the fact that they connect two distinct geometric aspects of Banach spaces, namely ball separation properties and the geometry of the dual unit ball. In this sense, Property (II) serves as a natural and important weakening of MIP.
More recently, a further refinement of the characterization of Property (II), describing the geometry of all points on the dual unit sphere, was obtained by authors in \cite{BS}: $X$ has Property (II) if and only if every point of the dual unit sphere is a semi $w^*$-Point of Continuity. 



Over the past few decades, considerable attention has been devoted to various generalizations of MIP, particularly to the problem of determining when members of given families of closed bounded convex sets (rather than all closed bounded convex sets) can be represented as intersections of closed balls. Early investigations focused on the families of compact convex sets \cite{S1, WZ, V}, weakly compact convex sets \cite{V, Z} and compact convex sets with finite affine dimension \cite{S}. Subsequently, in \cite{BG, CL, DC}, these questions were studied for compatible collection of bounded subsets of $X$.
Such generalized properties have often been characterized in the literature through the geometry of the dual unit ball, making this an active and evolving area of research. 

Such generalizations of Property (II), together with their geometric characterizations, appear to be absent from the existing literature.
Motivated by the above works, particularly \cite{BG, V}, we address this gap by introducing the following generalization of Property (II), which we call $\mathcal{A}$-Property (II).

\begin{definition}
Let $\mathcal{A}$ be a compatible collection of bounded subsets of $X$.
A Banach space $X$ is said to have $\mathcal{A}$-Property (II) if for every closed convex set $A\in \mathcal{A}$ and $\beta \geqslant 0$, $\overline{A+\beta B_X}$ is an intersection of closed convex hulls of finitely many balls.
\end{definition}
The main goal of this paper is to characterize $\mathcal{A}$-Property (II) in terms of the geometry of the dual unit ball.
We shall now
outline the contents of the paper. In Section 2, we recall the necessary notation and preliminaries needed in the sequel. In Section 3, we obtain our main result. In Section 4, we discuss a special case of $\mathcal{A}$-Property (II), which yields an interesting weaker geometric property
of the dual unit ball.
Finally, in Section 5, we present some further consequences and conclude with an open problem.

 \section{Notations and preliminaries}
 
 Throughout this paper, $X$ denotes a real Banach space. The closed unit ball of $X$ and the unit sphere of $X$ are denoted by $B_X$ and $S_X$, respectively. For any subset $C\subset X,$ we adopt the following standard notations: $\overline{C}$ for the closure of $C$, co($C$) for the convex hull of $C$, $|\mathrm{co}|(C)$ for the absolutely convex hull of $C$, and diam $(C)$ for the diameter of $C$. Moreover, the cone generated by $C$ is denoted by cone$(C) = \{\lambda h : \lambda > 0, h \in C\}$. The ball 
 centered at $x\in X$ with radius $r>0$ is denoted by $B(x,r)$.
 
A $w^*$-slice of $B_{X^*}$ is a set of the form
$$S(B_{X^*}, x, \alpha) = \{f \in B_{X^*} : f(x) > \sup\limits_{f\in B_{X^*}} f(x) - \alpha \}$$
where $x\in X$ and $\alpha > 0.$ Throughout, we assume without loss of generality that $\|x\| = 1$. 
 Note that every $w^*$-slice is  $w^*$-open, and the collection of all $w^*$-slices forms a subbase for the $w^*$-toplogy on $B_{X^*}$. Moreover, by \cite[Lemma 4.2]{CL}, the $w^*$-toplogy on $B_{X^*}$ admits a neighbourhood base consisting of sets of the form $$\{f\in B_{X^*} : f(x_i)>\gamma_i>0, \ i=1,2,\ldots,n\}$$
where $x_1,x_2,\ldots,x_n$ in $X$.


For any bounded subset $A \subset X$, we define the function $\|\cdot\|_A : X^* \to \mathbb{R}$ by $$\|f\|_A = \sup\{|f(x)| : x \in A\}, \quad \text{for } f \in X^*.$$
It is straightforward to verify that $\|\cdot\|_A$ induces a seminorm on $X^*$. Under this seminorm, the diameter of a subset $B \subset X^*$ is given by $\text{diam}_A(B) = \sup\{\|f - g\|_A : f, g \in B\}$. Furthermore, the open ball in $X^*$ centered at $f \in X^*$ with radius $\varepsilon > 0$, generated by the seminorm $\|\cdot\|_A$, is denoted by
$B_A(f, \varepsilon) = \{g \in X^* : \|g - f\|_A < \varepsilon\}.$


\begin{definition}
\cite{CL}
A collection $\mathcal{A}$ of bounded subsets of $X$ is said to be compatible  if for each $A\in \mathcal{A}$, we have
\begin{enumerate}
\item  $C\in \mathcal{A}$ for all $C\subset A.$
\item  $A+x, A\bigcup \{x\} \in \mathcal{A}$ for all $x\in X.$
\item  $\overline{|\mathrm{co}|}(A)\in \mathcal{A}.$ 
\end{enumerate}
\end{definition}

 \begin{definition}
  A point $f\in S_{X^*}$ is called 
 \begin{enumerate}
     \item \cite{DU} $w^*$-Point of Continuity ($w^*$-PC)  of $B_{X^*}$ if for each $\varepsilon >0$, there exists a $w^*$-open set $W$ of $B_{X^*}$ such that $f\in W \subset B(f, \varepsilon)$.
     \item 
     \cite{BS}
     semi $w^*$-Point of Continuity (semi $w^*$-PC) of $B_{X^*}$ if for each $\varepsilon >0$, there exists a $w^*$-open set $W$ of $B_{X^*}$ such that  $W \subset B(f, \varepsilon)$.
 \end{enumerate}
\end{definition}


We conclude this section with the following result due to Phelps, which will play a significant role in the subsequent sections.
\begin{lemma}\label{phelp's lemma}
\cite{P} 
Let $X$ be a normed space and $f,g \in S_{X^*}$. Consider $A=\{x\in B_X:f(x)>\frac{\varepsilon}{2}\}$ where $0<\varepsilon<1.$ If $\inf g(A)>0,$ then $\|f-g\|<\varepsilon.$
\end{lemma}

\section{Main result}
Before presenting our main result, we introduce a generalized notion of the semi $w^*$-PC of $B_{X^*}$ and establish its local geometric characterization in terms of ball separation.

\begin{definition}
Let $\mathcal{A}$ be a collection of bounded subsets in $X$. We say that $f\in S_{X^*}$ is
  $\mathcal{A}$-semi Point of Continuity ($\mathcal{A}$-semi PC) of $B_{X^*}$ if for each $A\in \mathcal{A}$ and $\varepsilon >0$, there exists a $w^*$-open set $W$ of $B_{X^*}$ such that $W\subset B_A(f, \varepsilon)$.
 \end{definition}
 


 
 \begin{theorem}\label{Mainastrsemipc}
Let $\mathcal{A}$ be a compatible collection of bounded subsets of $X$ and $f_0\in S_{X^*}$. The following are equivalent.
\begin{enumerate}
\item $f_0$ is $\mathcal{A}$-semi PC of $B_{X^*}$.
\item For any $A\in \mathcal{A}$, $\beta \geqslant 0$ and $x_0\in X$, if $\inf f_0(A+\beta B_X)>f_0(x_0)$,
then there are finitely many balls $B_1,\ldots,B_n$ in $X$ such that $(A+\beta B_X) \subset \overline{\mathrm{co}}\bigcup\limits_{i=1}^{n} B_i$ and $x_0\notin \overline{\mathrm{co}}\bigcup\limits_{i=1}^{n} B_i$.
\end{enumerate}
\end{theorem}

\begin{proof}
 (i) $\Longrightarrow$ (ii) : 
Let $A\in \mathcal{A}$, $\beta \geqslant 0$ and $x_0\in X$ such that $\inf f_0(A+\beta B_X)>f_0(x_0)$. Since $\mathcal{A}$ is a compatible collection, we may assume without loss of generality that $x_0=0$ and $A\subset B_X.$ 

Indeed, choose 
 $m>0$ such that $\frac{1}{m} (A-x_0)\subset B_X$. 
 Then $\inf f_0(\frac{1}{m} (A-x_0)+\frac{\beta}{m} B_X)>0$.  
 If there exist balls $B(z_1,r_1), \ldots, B(z_n,r_n)$ in $X$ such that 
 $(\frac{1}{m} (A-x_0)+\frac{\beta}{m} B_X) \subset \overline{\mathrm{co}} (\bigcup\limits_{i=1}^{n} B(z_i,r_i))$ 
 and $0\notin \overline{\mathrm{co}} (\bigcup\limits_{i=1}^{n} B(z_i,r_i))$, then $(A + \beta B_X)\subset \overline{\mathrm{co}} (\bigcup\limits_{i=1}^{n} B(x_0+mz_i,mr_i))$ and $x_0\notin \overline{\mathrm{co}} (\bigcup\limits_{i=1}^{n} B(x_0+mz_i,mr_i))$.

Also let $\inf f_0(A+\beta B_X)=2\alpha>0$. 
Since $f_0$ is $\mathcal{A}$-semi PC of $B_{X^*}$, there exists a $w^*$-open set 
$$V=\{f\in B_{X^*} : f(x_i)>\gamma_i>0, x_i\in X, \ i=1,2,\ldots,n\}$$ of $B_{X^*}$
(where $x_1,\ldots,x_n\in S_X$) such that $V\subset B_A(f_0, \alpha)$. Then 
\begin{equation}\label{Mainm1}
\inf f_1(A)\geqslant\alpha + \beta \quad \forall f_1\in V.
\end{equation}
Indeed, given $f_1\in V$, we have 
$$\inf f_1(A)  \geqslant  \inf f_0(A) - \|f_0-f_1\|_A \geqslant 2\alpha + \beta -\alpha= \alpha + \beta >0.$$

Choose $f\in V$. Then for each $i=1,\ldots,n$, we have $f(x_i)>\gamma_i$ and therefore we can choose $\alpha_i>0$ such that $f(x_i)>\alpha_i> \gamma_i>0.$ We also select $m\in \mathbb{N}$ with $\alpha_i-\frac{1+\beta}{m}> \gamma_i$ for all $i=1,\ldots,n$. 
Finally, choose $x_0\in \alpha B_X$ such that $f(x_0)>0$, and hence choose $0<\eta < f(x_0)$.

 Consider $$K=\overline{\mathrm{co}} \{B(mx_1, m\alpha_1) \bigcup \ldots \bigcup B(mx_k, m\alpha_k)\bigcup B(x_0,\eta) \}.$$ 
 Observe that, for all $i=1,\ldots,n$, 
\begin{equation}\notag 
\begin{split}
&\inf f(B(mx_i, m\alpha_i))\geqslant f(mx_i)- m\alpha_i>0,\\ 
\text{and} \ \ 
  & \inf f(B(x_0,\eta))= f(x_0)- \eta>0.
  \end{split}
\end{equation} 
 Therefore,  $\inf f(K)\geqslant \min\{f(x_0)- \eta, f(mx_1)- m\alpha_1, \ldots, f(mx_n)- m\alpha_n\}  >0.$ 
 Hence, $0\notin K.$
 
Lastly, we claim that  $A+\beta B_X\subset K.$ 
If not, there exists $y_0\in (A+\beta B_X)\setminus K.$ By Hahn-Banach separation theorem there exists $h\in S_{X^*}$ such that $\inf h(K) > h(y_0)\geqslant -(1+\beta).$ Therefore, for each $i=1,\ldots,n$, $$h(mx_i)-m\alpha_i = \inf h (B(mx_i,m\alpha_i))\geqslant \inf h(K)>-(1+\beta)$$
That gives, $h(x_i)>\alpha_i-\frac{1+\beta}{m}>\gamma_i$ for all $i=1,\ldots,n$ and hence $h\in V.$ Thus from ($\ref{Mainm1}$), it follows that,
\begin{equation}\label{Mainchap 6 equ}
 \inf h(K) > h(y_0)\geqslant \inf h(A+\beta B_X)\geqslant\alpha +\beta-\beta=\alpha.
\end{equation}
Again,
\begin{equation}\label{Mainchap 6 equation}
 \inf h(K)\leqslant h(x_0)\leqslant \|x_0\|\leqslant \alpha.
\end{equation}
Therefore from ($\ref{Mainchap 6 equ}$) and ($\ref{Mainchap 6 equation}$) we get a contradiction. Hence, our claim is true. 

 (ii) $\Longrightarrow$ (i).  
 Let $A\in \mathcal{A}$ and $\varepsilon>0.$ Choose $x_0\in X$ such that 
\begin{equation} \label{Maineqstrong}
 \|x_0\|=\sup\{\|a\|: a\in A\}+4\varepsilon, \ \ f_0(x_0)>\|x_0\|-\varepsilon.
 \end{equation}
 Let $K=\overline{|\mathrm{co}|}(A\bigcup \{x_0\})$, $K_{\varepsilon}=\{x\in K : f_0(x)\geqslant 2 \varepsilon\}$ and $\beta = \varepsilon$. Then $\inf f_0(K_{\varepsilon}+\beta B_X)\geqslant \varepsilon >0$. Thus there are finitely many balls $B_1=B(x_1,r_1),\ldots, B_n=B(x_n,r_n)$ in $X$ such that $$K_{\varepsilon}+\beta B_X \subset \overline{\mathrm{co}}\bigcup\limits_{i=1}^{n} B_i \quad \text{and} \quad 0\notin \overline{\mathrm{co}}\bigcup\limits_{i=1}^{n} B_i.$$ 
 By Hahn-Banach separation theorem there exists $g\in S_{X^*}$ such that $\inf g(\overline{\mathrm{co}} ( \bigcup\limits_{i=1}^{n} B_i))>0.$ 
 Consider $w^*$-open set 
  $$V=\bigcap\limits_{i=1}^{n} S(B_{X^*}, \frac{x_i}{\|x_i\|}, 1-\frac{r_i}{\|x_i\|}).$$ 
 Observe that $g\in V$ and hence $V$ is nonempty.

Fix $h\in V.$ Then $h(x_i)>r_i$ for all $i=1,\ldots,n$. Therefore, 
 \begin{equation}\notag
 \begin{split}
 \inf h(K_{\varepsilon}+\beta B_X) \geqslant \inf h (\overline{\mathrm{co}} (\bigcup\limits_{i=1}^{n} B_i))
 \geqslant \min\{\inf h(B_1), \ldots, \inf h(B_n) \}\hspace{1.5 cm}\\
  \geqslant \min\{h(x_1) - r_1, \ldots, h(x_n) - r_n\}\hspace{1 cm}\\
  >0.\hspace{5.6 cm}
 \end{split}
 \end{equation}
 It follows that $\inf h(K_{\varepsilon})>\beta$.
Applying Lemma $\ref{phelp's lemma}$ to the space Y = span $K$ with unit ball $K$, we have 
\begin{equation}\label{Mainbhp ball eq1}
\|\frac{f_0}{\|f_0\|_K}-\frac{h}{\|h\|_K}\|_K<4\frac{\varepsilon}{\|f_0\|_K}.
\end{equation} 
Also, observe that $2\varepsilon \frac{x_0}{f_0(x_0)} \in K_{\varepsilon}$. Hence $\beta < \inf h(K_{\varepsilon})< h(2\varepsilon \frac{x_0}{f_0(x_0)})$. Thus $h(x_0)>\frac{f_0(x_0) \beta}{2\varepsilon}=f_0(x_0)$.

From ($\ref{Maineqstrong}$), it follows that 
\begin{equation}\label{Maineqqq1}
\|x_0\|-\varepsilon <f_0(x_0) \leqslant \|f_0\|_K\leqslant\|x_0\|
\end{equation}
\begin{equation}\label{Maineqqq2}
\begin{split}
\|x_0\|\geqslant \|h\|_K\geqslant \sup h(K_{\varepsilon})\geqslant h(x_0)=f_0(x_0)>\|x_0\|-\varepsilon.
\end{split}
\end{equation}
Thus from ($\ref{Maineqqq1}$) and ($\ref{Maineqqq2}$), we have $|\|f_0\|_K-\|h\|_K|<\varepsilon$.
Then 
\begin{equation}\notag
\begin{split}
\|f_0-h\|_A \leqslant \|f_0-h\|_K
\leqslant \|f_0-\frac{\|f_0\|_K}{\|h\|_K} h\|_K + \|\frac{\|f_0\|_K}{\|h\|_K} h - h\|_K\\
<4\varepsilon + |\|f_0\|_K - \|h\|_K|<5\varepsilon.\hspace{1.3 cm}
\end{split}
\end{equation}
Therefore, $h\in B_A(f_0,5\varepsilon)$. Hence, $V\subset B_A(f_0,5\varepsilon).$
\end{proof}

The following theorem is the main result of this paper.

\begin{theorem}\label{Mainthastrcor}
Let $\mathcal{A}$ be a compatiable collection of bounded subsets of $X$. Then $X$ has  $\mathcal{A}$-Property (II) if and only if every $f$ in $S_{X^*}$ is $\mathcal{A}$-semi PC of $B_{X^*}$.
 \end{theorem}

\begin{proof}
First part follows directly from Theorem $\ref{Mainastrsemipc}$.

Conversely, suppose that every $f\in S_{X^*}$ is $\mathcal{A}$-semi PC of $B_{X^*}$.
Let $A\in \mathcal{A}$ be a closed convex set, $\beta \geqslant 0$ and $x_0\notin \overline{A+\beta B_X}$. Without loss of generality, we may assume that $x_0=0$.
By the Hahn-Banach separation theorem there exists $f \in S_{X^*}$ such that $\inf f (\overline{A+\beta B_X}) > 0$. 
Since $f$ is $\mathcal{A}$-semi PC of $B_{X^*}$ and $\inf f (A+\beta B_X) > 0$, then Theorem $\ref{Mainastrsemipc}$ yields finitely many balls $B_1,\ldots,B_n$ in $X$ with $A+\beta B_X \subset \overline{\mathrm{co}}\bigcup\limits_{i=1}^{n} B_i$ and $0\notin \overline{\mathrm{co}}\bigcup\limits_{i=1}^{n} B_i$. Hence, $\overline{A+\beta B_X} \subset \overline{\mathrm{co}}\bigcup\limits_{i=1}^{n} B_i$ and $0\notin \overline{\mathrm{co}}\bigcup\limits_{i=1}^{n} B_i.$
\end{proof}

\section{A special case of $\mathcal{A}$-Property (II)}
In this section, we investigate a particular case of $\mathcal{A}$-Property (II), corresponding to $\beta=0$. In this setting, every closed bounded convex set $A\in \mathcal{A}$ can be represented as the intersection of closed convex hulls of finitely many balls. We show that even under this weaker framework one still obtains a complete characterization in terms of geometry of dual unit ball (see Proposition $\ref{sp case}$). To this end, we introduce the following notion:
\begin{definition}
Let $\mathcal{A}$ be a collection of bounded subsets in $X$. A point $f\in S_{X^*}$ is said to be weak $\mathcal{A}$-semi Point of Continuity (weak $\mathcal{A}$-semi PC) of $B_{X^*}$ if for each $A\in \mathcal{A}$ and $\varepsilon >0$, there exists a $w^*$-open set $W$ of $B_{X^*}$ such that $W\subset \mathrm{cone}(B_A(f, \varepsilon))$.
\end{definition}

\begin{proposition}\label{asemipc}
Let $\mathcal{A}$ be a compatible collection of bounded subsets of $X$ and $f_0\in S_{X^*}$. The following are equivalent.
\begin{enumerate}
\item $f_0$ is weak $\mathcal{A}$-semi PC of $B_{X^*}$.
\item For any $A\in \mathcal{A}$ and $x_0\in X$, if $\inf f_0(A)>f_0(x_0)$, then there are finitely many balls $B_1,\ldots,B_n$ in $X$ such that $A \subset \overline{\mathrm{co}}\bigcup\limits_{i=1}^{n} B_i$ and $x_0\notin \overline{\mathrm{co}}\bigcup\limits_{i=1}^{n} B_i$.
\end{enumerate}
\end{proposition}

\begin{proof}
 (i) $\Longrightarrow$ (ii). 
Let $A\in \mathcal{A}$. Since $\mathcal{A}$ is a compatible collection, then as in Theorem $\ref{Mainastrsemipc}$ (i) $\Rightarrow$ (ii), we may assume without loss of generality that $x_0=0$ and $A\subset B_X$. 

Now suppose $\inf f_0(A)=2\alpha>0$. Then $\alpha < \frac{1}{2}$. Since $f_0$ is weak $\mathcal{A}$-semi PC of $B_{X^*}$, there exists a $w^*$-open neighbourhood 
$$V=\{f\in B_{X^*} : f(x_i)>\gamma_i>0, x_i\in X, \ i=1,2,\ldots,n\}$$ of $B_{X^*}$
(where $x_1,\ldots,x_n\in S_X$) such that $V\subset \mathrm{cone}(B_A(f_0, \alpha))$. Then 
\begin{equation}\label{m1}
\inf f_1(A)>0 \ \forall f_1\in V.
\end{equation}
Indeed, given $f_1\in V$, we have $f_1=\lambda g_1$ for some $\lambda >0$ and $g_1\in B_A(f_0, \alpha)$. Hence $$\inf f_1(A) = \lambda \inf g_1(A) \geqslant \lambda (\inf f_0(A) - \|f_0-g_1\|_A) \geqslant \lambda (2\alpha -\alpha)= \lambda \alpha >0.$$

Choose $f\in V$. Proceeding as in the proof of Theorem $\ref{Mainastrsemipc}$ (i) $\Rightarrow$ (ii), we obtain 
$$K=\overline{\mathrm{co}} \{B(mx_1, m\alpha_1) \bigcup \ldots \bigcup B(mx_k, m\alpha_k)\bigcup B(x_0,\eta) \} \quad \text{with} \quad 0\notin K,$$
where $\alpha_i>0$ and $m\in \mathbb{N}$ are such that $f(x_i)>\alpha_i> \gamma_i>0$ and $(1-\frac{1}{m})\alpha_i-\frac{1}{m}> \gamma_i$ for all $i=1,\ldots,n$. Also $x_0\in (\frac{m}{m-1} d(0,A))B_X$ such that  $0<\eta < f(x_0)$.


 Lastly, we claim that  $A\subset K.$ 
If not, there exists $y_0\in A\setminus K.$ By Hahn-Banach separation theorem there exists $h\in S_{X^*}$ such that $\inf h(K) > h(y_0)\geqslant -1.$ Arguing exactly as in the proof of Theorem $\ref{Mainastrsemipc}$ (i) $\Rightarrow$ (ii)  $h\in V.$ 
Thus from ($\ref{m1}$), it follows that, $h(y_0)\geqslant \inf h(A)>0.$
Therefore, $h(x_i)>\alpha_i$ for all $i=1,\ldots,n$.

Fix $x\in A$ and $f_x\in S_{X^*}$ such that $f_x(x)=\|x\|$. Then for all $i=1,\ldots,n$,
\begin{equation}\notag
\begin{split}
((1-\frac{1}{m})h-\frac{1}{m} f_x)(x_i)>(1-\frac{1}{m}) \alpha_i -\frac{1}{m} > \gamma_i
\end{split}
\end{equation}
Thus $(1-\frac{1}{m})h-\frac{1}{m} f_x\in V \subset \mathrm{cone}(B_A(f_0, \alpha)).$ Then from ($\ref{m1}$), $$((1-\frac{1}{m})h-\frac{1}{m} f_x)(x)\geqslant \inf  ((1-\frac{1}{m})h-\frac{1}{m} f_x)(A)>0.$$
Then $h(x)> \frac{\|x\| m}{m-1}\geqslant \frac{d(0,A) m}{m-1}.$ It follows that 
\begin{equation}\label{chap 6 equ}
\inf h(K)>h(y_0)> \frac{d(0,A) m}{m-1}.
\end{equation}
Again,
\begin{equation}\label{chap 6 equation}
 \inf h(K)\leqslant h(x_0)\leqslant \|x_0\|\leqslant \frac{d(0,A) m}{m-1}.
\end{equation}
Therefore from ($\ref{chap 6 equ}$) and ($\ref{chap 6 equation}$) we get a contradiction. Hence, our claim is true.

(ii) $\Longrightarrow$ (i). 
 Let $A\in \mathcal{A}$ and $\varepsilon>0.$ Choose $x_0\in X$ such that $$\|x_0\|=\sup\{\|a\|: a\in A\}+2\varepsilon, \ \ f_0(x_0)>\|x_0\|-\varepsilon.$$
 Let $K=\overline{|\mathrm{co}|}(A\bigcup \{x_0\})$ and $K_{\varepsilon}=\{x\in K : f_0(x)\geqslant \varepsilon\}$. Then $\inf f_0(K_{\varepsilon})\geqslant \varepsilon >0$. Proceeding as in the proof of Theorem $\ref{Mainastrsemipc}$ (ii) $\Rightarrow$ (i), we obtain a nonempty $w^*$-open set  $V$ of $B_{X^*}$ such that for each $h\in V$, 
 \begin{equation}\label{bhp ball eq1}
\|\frac{f_0}{\|f_0\|_K}-\frac{h}{\|h\|_K}\|_K<2\frac{\varepsilon}{\|f_0\|_K}.
\end{equation} 
It follows that 
$$\|f_0-\frac{\|f_0\|_K}{\|h\|_K} h\|_A \leqslant \|f_0-\frac{\|f_0\|_K}{\|h\|_K} h\|_K <2\varepsilon.$$
Therefore, $h= \frac{\|h\|_K}{\|f_0\|_K} (\frac{\|f_0\|_K}{\|h\|_K} h)\in \mathrm{cone}(B_A(f_0,2\varepsilon))$. Hence, $V\subset \mathrm{cone}(B_A(f_0,2\varepsilon)).$Therefore, $h= \frac{\|h\|_K}{\|f_0\|_K} (\frac{\|f_0\|_K}{\|h\|_K} h)\in \mathrm{cone}(B_A(f_0,2\varepsilon))$. Hence, $V\subset \mathrm{cone}(B_A(f_0,2\varepsilon)).$

\end{proof} 


Suppose $\mathcal{A}$ is a compatible collection of bounded subsets of $X$, and let  $\tau_{\mathcal{A}}$ be the topology on $X^*$ generated by the seminorms $\{\|\cdot\|_A : A \in \mathcal{A}\}$. Under this topology, the sets $\{B_A(f, \varepsilon) : A\in \mathcal{A}, \varepsilon >0\}$ form a local base at 0.


\begin{proposition}\label{sp case}
Let $\mathcal{A}$ be a compatible collection of bounded subsets of $X$. The following are equivalent.
\begin{enumerate}
\item Every closed bounded convex set $A\in \mathcal{A}$ is an intersection of closed convex hulls of finitely many balls.
\item Every $f$ in $S_{X^*}$ is weak $\mathcal{A}$-semi PC of $B_{X^*}$.
\item The cone of weak $\mathcal{A}$-semi PC of $B_{X^*}$ is $\tau_\mathcal{A}$-dense in $X^*$.
\end{enumerate}
 \end{proposition}

\begin{proof}
(i) $\Longrightarrow$ (ii) and (ii) $\Longrightarrow$ (iii) are immediate from Proposition $\ref{asemipc}$.

(iii) $\Longrightarrow$ (i). Let $A\in \mathcal{A}$ be a closed convex set and $x_0\notin A$. Without loss of generality, we may assume that $x_0=0$.
By the Hahn-Banach separation theorem there exists $f \in S_{X^*}$ such that $\inf f (A) > 0$. Then there exist $\lambda > 0$ and weak $\mathcal{A}$-semi PC
$f_0 \in S_{X^*}$  such that
$\|f-\lambda f_0\|_A < \inf f (A) $.
Hence $\inf f_0 (A) > 0$. It follows from Proposition $\ref{asemipc}$ that there are finitely many balls $B_1,\ldots,B_n$ in $X$ with $A \subset \overline{\mathrm{co}}\bigcup\limits_{i=1}^{n} B_i$ and $0\notin \overline{\mathrm{co}}\bigcup\limits_{i=1}^{n} B_i$.
\end{proof}

\begin{remark}
The results analogous to Proposition $\ref{asemipc}$ and $\ref{sp case}$ in the setting of MIP were studied in \cite{DC}.
\end{remark}

\section{Remarks and Open problem}

The following generalization of $w^*$-PC was introduced by Chen and Lin \cite{CL}.
\begin{definition}
\cite{CL}
Let $\mathcal{A}$ be a collection of bounded subsets in $X$.  A point $f \in S_{X^*}$ is said to be $\mathcal{A}$ PC of $B_{X^*}$ if for each $A\in \mathcal{A}$ and $\varepsilon >0$, there exists a $w^*$-open set $W$ of $B_{X^*}$ such that $f\in W\subset B_A(f, \varepsilon)$.
\end{definition}

Let $\mathcal{K}$ denote the family of all compact convex subsets of $X$. Then we have the following. 

\begin{proposition}\label{cpt}
Every $f$ in $S_{X^*}$ is $\mathcal{K}$ PC of $B_{X^*}$.
\end{proposition}

\begin{proof}
Let $f_0\in S_{X^*}$. Also let $C$ be a compact set in $B_X$ and $\varepsilon >0$. Since $C$ is compact, then there exists $c_1,\ldots, c_n\in C$ such that $C\subset \bigcup\limits_{i=1}^{n} B(c_i,\varepsilon)$. Consider $w^*$-open set $W=\{f\in B_{X^*}: |f(c_i)-f_0(c_i)|<\varepsilon\}$ of $B_{X^*}$. Let $f,g \in W$. Fix $c\in C$. Then there exists $1\leqslant i \leqslant n$ such that $c\in B(c_i,\varepsilon)$. Thus
$$|f(c)-g(c)|\leqslant |f(c)-f(c_i)|+|f(c_i)-g(c_i)|+|g(c_i)-g(c)|<4\varepsilon.$$
Hence $f_0\in W\subset B_C(f_0,4\varepsilon)$.
\end{proof}

Since $\mathcal{K}$ PC is stronger than $\mathcal{K}$-semi PC,  Theorem $\ref{Mainthastrcor}$ immediately yields the following corollary.
\begin{corollary}
Let $X$ be a Banach space. Then for every compact convex set $K\subset X$ and $\beta \geqslant 0$, $\overline{K+\beta B_X}$ is an intersection of closed convex hulls of finitely many balls.
\end{corollary}

 Observe that when $\mathcal{A}$ is the collection of all bounded subsets of $X$, the notion of $\mathcal{A}$ PC (resp. $\mathcal{A}$-semi PC) coincide with $w^*$-PC (resp. semi $w^*$-PC). The following example shows that $w^*$-PC is strictly stronger than semi $w^*$-PC. Consequently, $\mathcal{A}$ PC is strictly stronger than $\mathcal{A}$-semi PC.

\begin{example}
Consider $X=l_1$. By \cite[Corollary 2.8]{SM}, $X$ admits an isometric embedding into a Banach space $Z$ satisfying Property (II). Consequently, every point of $S_{Z^*}$ is  semi $w^*$-PC of $B_{Z^*}$ \cite[Theorem 4.6]{BS}. However, it need not be the case that every point of $S_{Z^*}$ is $w^*$-PC of $B_{Z^*}$. Indeed, if every point of $S_{Z^*}$ were $w^*$-PC of $B_{Z^*}$, then, by the hereditary property, every point of $S_{X^*}$ would also be  $w^*$-PC of $B_{X^*}$. This would imply that $X$ is Asplund \cite[Theorem 1.0.12, Theorem 2.1.28]{B2}, which is a contradiction. Therefore, there exists a point of $S_{Z^*}$ that is  semi $w^*$-PC of $B_{Z^*}$ but not  $w^*$-PC of $B_{Z^*}$.
\end{example}

As mentioned in Section 1, it is known that a Banach space $X$ has Property (II) if and only if the set of $w^*$-PC of $B_{X^*}$ is dense in $S_{X^*}$ \cite{CL}, which is further equivalent to every point of $S_{X^*}$ being a semi $w^*$-PC of $B_{X^*}$ \cite{BS}. Compare this to the generalized version of Property (II): for a compatible collection $\mathcal{A}$ of bounded subsets of $X$, Theorem $\ref{Mainthastrcor}$ establishes that $X$ has $\mathcal{A}$ Property (II) if and only if every point on $S_{X^*}$ is $\mathcal{A}$-semi PC of $B_{X^*}$. Furthermore, observe that if the set of$\mathcal{A}$-PC of $B_{X^*}$ is $\tau_{\mathcal{A}}$-dense in $S_{X^*}$, then every point of $S_{X^*}$ is $\mathcal{A}$-semi PC of $B_{X^*}$. Hence, by Theorem $\ref{Mainthastrcor}$, $X$ has $\mathcal{A}$ Property (II). This naturally leads to the following question:

{\bf Question}: Let $\mathcal{A}$ be a compatible collection of bounded subsets of $X$. If $X$ has $\mathcal{A}$ Property (II), does it follow that the set of $\mathcal{A}$-PC of $B_{X^*}$ is $\tau_{\mathcal{A}}$-dense in $S_{X^*}$?


\begin{Acknowledgement}
          The first author is grateful to Professor Timothy Clarke, Chair, Department of Mathematics and Statistics, Loyola University, for his support and encouragement. She is also grateful to Professor Bahram Roughani, Associate Dean of the College of Arts and Sciences, Loyola University, for providing her with the Dean's supplemental grant during her travel in Summer 2025. 
      \end{Acknowledgement}

\end{document}